\documentclass[11pt]{amsart}
\usepackage[margin=1.35in]{geometry}
\usepackage{amscd,amsmath,amsxtra,amsthm,amssymb,stmaryrd,xr,mathrsfs,mathtools,enumerate,commath, comment, enumitem}
\usepackage{lipsum}

\newcommand\blfootnote[1]{%
  \begingroup
  \renewcommand\thefootnote{}\footnote{#1}%
  \addtocounter{footnote}{-1}%
  \endgroup
}
\usepackage{stmaryrd}
\usepackage{xcolor}
\usepackage{commath}
\usepackage{comment}
\usepackage{tikz-cd}
\usepackage{longtable} 
\usepackage{pdflscape} 
\usepackage{booktabs}
\usepackage{hyperref}
\definecolor{vegasgold}{rgb}{0.77, 0.7, 0.35}
\definecolor{darkgoldenrod}{rgb}{0.72, 0.53, 0.04}
\definecolor{gold(metallic)}{rgb}{0.83, 0.69, 0.22}
\hypersetup{
 colorlinks=true,
 linkcolor=darkgoldenrod,
 filecolor=brown,      
 urlcolor=gold(metallic),
 citecolor=darkgoldenrod,
 }

\usepackage[all,cmtip]{xy}

\DeclareFontFamily{U}{wncy}{}
\DeclareFontShape{U}{wncy}{m}{n}{<->wncyr10}{}
\DeclareSymbolFont{mcy}{U}{wncy}{m}{n}
\DeclareMathSymbol{\Sh}{\mathord}{mcy}{"58}
\usepackage[T2A,T1]{fontenc}
\usepackage[OT2,T1]{fontenc}

\newtheorem{theorem}{Theorem}[section]
\newtheorem{lemma}[theorem]{Lemma}

\newtheorem*{theorem*}{Theorem}
\newtheorem*{ass*}{Assumption}
\newtheorem{definition}[theorem]{Definition}
\newtheorem{corollary}[theorem]{Corollary}
\newtheorem{remark}[theorem]{Remark}

\newtheorem{proposition}[theorem]{Proposition}
\newtheorem{question}[theorem]{Question}

\newcommand{\Z}{\mathbb{Z}}
\newcommand{\p}{\mathfrak{p}}
\newcommand{\Q}{\mathbb{Q}}
\newcommand{\F}{\mathbb{F}}

\newcommand{\cO}{\mathcal{O}}

\newcommand{\op}[1]{\operatorname{#1}}

\newcommand\mtx[4] { \left( {\begin{array}{cc}
 #1 & #2 \\
 #3 & #4 \\
 \end{array} } \right)}

\numberwithin{equation}{section}

\begin{document}

\title[Galois representations over function fields]{Galois representations over function fields that are ramified at one prime}

\author[A.~Ray]{Anwesh Ray}
\address[Ray]{Chennai Mathematical Institute, H1, SIPCOT IT Park, Kelambakkam, Siruseri, Tamil Nadu 603103, India}
\email{ar2222@cornell.edu}
\blfootnote{Corresponding author email: ar2222@cornell.edu}


\maketitle

\begin{abstract}
 Let $\mathbb{F}_q$ be the finite field with $q$ elements, $F:=\mathbb{F}_q(T)$ and $F^{\operatorname{sep}}$ a separable closure of $F$. Set $A$ to denote the polynomial ring $\F_q[T]$. Let $\p$ be a non-zero prime ideal of $A$, and $\mathcal{O}$ be the completion of $A$ at $\mathfrak{p}$. Given any integer $r\geq 2$, I construct a Galois representation $\rho:\operatorname{Gal}(F^{\op{sep}}/F)\rightarrow \operatorname{GL}_r(\cO)$ which is unramified at all non-zero primes $\mathfrak{l}\neq \mathfrak{p}$ of $A$, and whose image is a finite index subgroup of $\operatorname{GL}_r(\cO)$. Moreover, if the degree of $\p$ is $1$, then $\rho$ is also unramified at $\infty$. 
\end{abstract}

\section{Introduction}
\par Given a global field $F$, let $F^{\op{sep}}$ be a choice of separable closure of $F$, and set $\op{G}_F:=\op{Gal}(F^{\op{sep}}/F)$. Although class-field theory provides us with a complete description of the maximal abelian quotient of $\op{G}_F$, the structure of $\op{G}_F$ is shrouded in mystery. A systematic approach to study questions at the heart of number theory is to classify representations of $\op{G}_F$. When $F$ is a number field, such representations are known to arise from geometry, for instance, from abelian varieties and certain automorphic representations. On the other hand, when $F$ is a global function field over a finite field, such representations are known to arise from Drinfeld modules and Drinfeld modular forms. In this article, the field $F$ is taken to be rational function field $\F_q(T)$, where $\F_q$ is a finite field with $q$-elements. This is the function field of the projective line $\mathbb{P}^1_{/\F_q}$, and is the function field analogue of the field of rational numbers $\Q$. Set $A$ to denote the polynomial ring $\F_q[T]$. 
\par The inverse Galois problem for $F$ can be interpreted as follows: given a finite group $G$, does there exist a surjective homomorphism $\varrho: \op{G}_F\rightarrow G$? A stronger variant of this question is to ask for the minimal number $n$, such that there exists a surjection $\varrho:\op{G}_F\rightarrow G$ that is unramified away from $n$ primes. Abhyankar \cite{abhyankar} studied related questions for finite Galois covers of the projective line (over a finite field), with specified ramification. In this context, celebrated \emph{Abhyankar's conjecture} gives a criterion for when a finite Galois cover of the projective line exists ramified at a prescribed number of points. This conjecture was resolved by Raynaud \cite{Raynaud} and Harbater \cite{Harbater}. 
\par In this article, I study a variant of this question, where instead of a finite group $G$, I consider the groups of the form $\op{GL}_r(\cO)$, where $\cO$ is the completion of $A$ at a non-zero prime ideal. Since such homorphisms arise as Galois representations from structures in function field arithmetic, this setting is of natural interest. More precisely, I consider the following question. 
\begin{question}\label{main question}
    Let $q$ be a power of a prime $p$, and $F$ be the rational function field $\F_q(T)$. Given $r\in \Z_{\geq 2}$ and a non-zero prime $\p$ of $A$, does there exist a representation \[\rho:\op{G}_F\rightarrow \op{GL}_r(\cO)\] which is unramified away from $\{\p\}$, and such that the image of $\rho$ is a finite-index subgroup of $\op{GL}_r(\cO)$?
\end{question}
\par Over $\Q$, the analogue of the above question was considered by Greenberg \cite{greenberg2016galois}. Let $p$ be an odd \emph{regular prime number} (i.e., $p$ does not divide the class number of $\Q(\mu_p)$) and $n$ be a positive integer. Greenberg shows that if $p\geq 4\lfloor n/2\rfloor+1$, then there is a Galois representation
$\rho:\op{Gal}(\bar{\Q}/\Q)\rightarrow \op{GL}_n(\Z_p)$ that is unramified at all primes $\ell\notin \{p, \infty\}$. Refinements and generalizations of Greenberg's result has been obtained in various contexts, cf. \cite{CornutRay, katz2020note, RayJNT, Tang, Maletto, Maire, RayTAMS}.

In this article, I obtain an affirmative answer to the Question \ref{main question}. Moreover, one has an explicit description of how the inertia group at $\p$ acts. Before stating the main result, I remind the reader that the degree of $\p$ is defined to be the $\F_q$-dimension of $A/\p$, or equivalently, the degree of an irreducible monic polynomial that generates $\p$. 
\begin{theorem}\label{main thm}
    Let $\p$ be a non-zero prime of $A$ and $\cO$ the completion $A_\p$ with uniformizer $\varpi$. Let $r\geq 2$ be an integer. There is a Galois representation 
    \[\rho:\op{G}_F\rightarrow \op{GL}_r(\cO),\] such that 
    \begin{itemize}
        \item $\rho$ is unramified at all non-zero primes $\mathfrak{l}\neq \p$ of $A$.
        \item If $\op{deg}\p=1$, then $\rho$ is also constructed to be unramified at $\infty:=(1/T)$.
        \item The image of $\rho$ has finite index in $\op{GL}_r(\cO)$.
        \item The restriction of $\rho$ to the inertia group $\op{I}_{\p}$ is of the form
        \[\rho_{|\op{I}_{\p}}=\mtx{\chi}{\ast}{0}{\op{Id}_{r-1}}=\begin{pmatrix}
  \chi & \ast & \ast & \cdots & \ast & \ast \\
   & 1 & 0 & \cdots & 0 & 0\\
    & & 1 & 0 & \cdots & 0 \\
    & &  & \ddots & \vdots & \vdots \\  
    & &  &    & 1 & 0 \\
    & &  &    &  & 1 &   \\
   \end{pmatrix},\] where $\chi$ is a totally ramified character. Thus there is an $(r-1)$-dimensional $\op{Gal}(\bar{F}_{\p}/F_\p)$-quotient on which $\op{I}_{\p}$ acts trivially.
\end{itemize}
\end{theorem}
\par The Galois representation constructed does also have the property of being geometric in a suitable sense. In fact, it arises from a Drinfeld module $\phi$ of rank $r$ over $F$. It is shown that the endomorphism ring of $\phi$ is naturally isomorphic to $A$. The open image result is then deduced from a theorem of Pink and R\"utsche \cite{PinkRutsche}. This Drinfeld module has good reduction at all other primes $\mathfrak{l}$ of $A$ and has stable reduction at $\p$. The Galois representation has very little ramification at $\p$ and is also unramified at $\infty$ if the degree of $\p$ is $1$. These properties are quite remarkable. The argument showing that the Galois representation is unramified at $\infty$ is delicate and does not readily generalize to higher degrees. 
\section{Preliminary notions}
\par Let $p$ be a prime and $q$ be a power of $p$. Denote by $\F_q$ the finite field with $q$ elements, and $F=\F_q(T)$ be the rational function field over $\F_q$. Let $A$ be the polynomial ring $\F_q[T]$, and $v_\infty$ be the valuation normalized by $v_\infty(T)=-1$. Denote by $\Omega_F$ the isomorphism classes of valuations $v$ of $F$ such that $v\neq v_\infty$. Given a non-zero prime ideal $\p$ and $a\neq 0$ in $A$, set $\op{ord}_{\p}(a)$ to denote the highest power $n$, such that $\p$ divides $a A$. For $a/b\in F^\times$, set 
\[\op{ord}_{\p}(a/b):=\op{ord}_{\p}(a)-\op{ord}_{\p}(b).\] For $v\in \Omega_F$, there is a non-zero prime ideal $\p$ of $A$ such that
\[v=v_{\p}:=\op{ord}_{\p}(\cdot). \]
Denote by $f_{\p}\in A$ the unique monic polynomial generator of $\p$. Given a valuation $v\in \Omega_F$, take $F_v$ to be the completion of $F$ at $v$. For $v=v_\p$, set $A_{v}:=A_{\p}$ be the completion of $A$ at $\p$. Then, $F_v:=F_\p$ is fraction field of $A_v$. Denote by $\kappa_v$ the residue field of $A_v$. Let $\bar{F}$ be the separable closure of $F$, and $\op{G}_F:=\op{Gal}(\bar{F}/F)$ the absolute Galois group of $F$.
 \par Let $K$ be a field equipped with a homomorphism of rings $\gamma: A\rightarrow K$. Such a field is referred to as an \emph{$A$-field}. In practice, $K$ will be $F$, $F_v$ or the residue field $\kappa_v$ for some $v\in \Omega_F$. Such a field is referred to as an $A$-field. The $A$-characteristic $\op{char}_A(K):=0$ if $\gamma$ is injective. Otherwise, if $\gamma$ is not injective, $\op{char}_A(K):=\op{ker}\gamma$. Say that $K$ has \emph{generic characteristic} if $\op{char}_A(K)=0$. 
 
 \par Let $K\{\tau\}$ be the non-commutative $\F_q$-algebra of twisted polynomials over $K$. Elements of $K\{\tau\}$ are polynomials of the form 
\[f(\tau)=\sum_{i=0}^d a_i \tau^i,\] where $a_i\in K$. Addition is defined \emph{term-wise} as follows
\[\begin{split}& \sum_i a_i \tau^i+\sum_i b_i \tau^i:=\sum_i(a_i+b_i) \tau^i,\\
& (a\tau^i) (b \tau^j):=ab^{q^i} \tau^{i+j}.\end{split}\]
The above relations uniquely determine the $\F_q$-algebra structure on $K\{\tau\}$. Let $K[x]$ be the polynomial ring over $K$. A polynomial $f(x)\in K[x]$ is \emph{additive} if the following relation holds in $K[x,y]$ 
\[f(x+y)=f(x)+f(y).\] The polynomial $f$ is $\F_q$-linear if it is additive and for all $\alpha\in \F_q$, one has that 
\[f(\alpha x)=\alpha f(x). \]
Let $K\langle x\rangle $ consist of polynomials in $K[x]$ that are $\F_q$-linear. This set of polynomials is an $\F_q$-subspace, however it is not closed under multiplication. On the other hand, if $f, g\in K\langle x\rangle$, then, their composite $f\circ g$ is also $\F_q$-linear. With respect to addition and composition $K\langle x\rangle $ forms an $\F_q$-algebra. 

\par Given a twisted polynomial $f(\tau)=\sum_i a_i \tau^i$, set $f(x)$ to be the $\F_q$-linear polynomial defined by $f(x):=\sum_i a_i x^{q^i}$. This defines a map 
\[\Pi: K\{\tau\}\rightarrow K\langle x\rangle,\] sending $f(\tau)$ to $f(x)$, and $\Pi$ is an isomorphism of $\F_q$-algebras. 

\par A twisted polynomial $f(\tau)$ can be written as 
\[f(\tau)=a_h \tau^h+a_{h+1} \tau^{h+1}+\dots +a_d \tau^d,\] where $a_h, a_d\neq 0$. The numbers $h, d$ satisfy $h\leq d$, and are the \emph{height} and \emph{degree} of $f(\tau)$ respectively. I set 
\[\op{ht}_\tau(f):=h\text{ and }\op{deg}_\tau(f):=d.\]
Note that the degree of $f(x)$ has degree $q^{\op{deg}_\tau(f)}$ as a polynomial in $x$. Let $\partial:K\{\tau\}\rightarrow K$ be the \emph{derivative map} sending 
\[\sum_n a_n \tau\mapsto a_0.\] Thus, $\op{ht}(f)=0$ if and only if $\partial(f)\neq 0$. Observe that $\partial(f)$ is equal to the derivative $\frac{d f}{dx}$ (as a function of $x$).
\begin{definition}
    Let $r\geq 2$ be an integer and $K$ be an $A$-field. A \emph{Drinfeld module} of rank $r$ over $K$ is a homomorphism of $\F_q$-algebras 
    \[\phi: A\rightarrow K\{\tau\},\] sending $a\in A$ to $\phi_a\in K\{\tau\}$, such that 
    \begin{itemize}
        \item $\partial(\phi_a)=\gamma(a)$ for all $a\in A$; 
        \item $\op{deg}_\tau(\phi_a)=r\op{deg}_T(a)$. 
    \end{itemize}
    The second condition is equivalent to requiring that $\op{deg}_\tau(\phi_T)=r$.
\end{definition}
Let $\phi$ be a Drinfeld module, and $\bar{K}$ an algebraic closure of $K$. Then, $\bar{K}$ acquires a \emph{twisted $A$-module structure}. Given $a\in A$ and $x\in \bar{K}$, set $a\cdot_\phi x:=\phi_a(x)\in \bar{K}$. Denote the associated twisted $A$-module by $^{\phi}\bar{K}$ the resulting $A$-module. Denote by $K^{\op{sep}}\subseteq \bar{K}$ the separable closure of $K$ and $\op{G}_K:=\op{Gal}(K^{\op{sep}}/K)$ the absolute Galois group of $K$. When $K$ is of generic characteristic, $\phi_a(x)$ is a separable polynomial for all $a\in A$ (cf. \cite[Proposition 3.3.4]{papibook}). Let $\phi: A\rightarrow F\{\tau\}$ be a Drinfeld module of rank $r$ and $a\in A$ be a non-zero element. Here, $\gamma: A\rightarrow F$ is the natural inclusion. Since $\partial a=a\neq 0$, it follows that $\phi_a$ is a separable polynomial. Denote by $\phi[a]$ the subset of $\bar{F}$ consisting of roots of $\phi_a(x)$. For $y\in \phi[a]$, it is easy to see that for $b\in A$, $\phi_b(y)\in \phi[a]$. With this $A$-module structure $\phi[a]$ one has that
\[\phi[a]\simeq \left(A/a\right)^r, \] cf. \cite[Theorem 3.5.2]{papibook}. The Galois representation on $\phi[a]$ is denoted 
\[\rho_{\phi, a}: \op{G}_F\rightarrow \op{GL}_r(A/a).\]

\par Let $\p$ be a non-zero prime ideal of $A$, the $\p$-adic Tate-module is the inverse limit 
\[T_\p(\phi):=\varprojlim_n \phi[\p].\]  The Galois representation associated to $T_{\p}(\phi)$ is denoted 
\[\widehat{\rho}_{\phi, \p}: \op{G}_F\rightarrow \op{GL}_r(A_\p),\] and one may identify $\widehat{\rho}_{\phi, \p}$ with the inverse limit $\varprojlim_n \rho_{\phi, \p}$.

\par Let $\phi$ be a Drinfeld module over $F$ and $v\in \Omega_F$. I denote by $\phi_v$ the localized Drinfeld module over $F_v$. Let $R_v$ be the valuation ring of $F_v$, and $\mathcal{M}_v$ be its maximal ideal. Denote by $\kappa_v$ the residue field $R_v/\mathcal{M}_v$.

\begin{definition}One says that $\phi$ has \emph{stable reduction} at $v$ if there is a Drinfeld module $\psi$ over $F_v$ that is isomorphic to $\phi_v$ with coefficients in $R_v$, for which the reduction $\bar{\psi}$ is a Drinfeld module. The rank of $\bar{\psi}$ is referred to as the \emph{reduction rank} of $\phi$ at $v$. If the reduction rank is $r$, then $\phi$ has \emph{good reduction} at $v$. 
\end{definition} Given a Drinfeld modules $\phi$ and $\psi$ over an $A$-field $K$, a \emph{morphism} $u\in \op{Hom}_K(\phi, \psi)$ is an element $u\in K\{\tau\}$ such that the relationship $u \phi_a=\psi_a u$ holds for all $a\in A$. The Endomorphism ring of $\phi$ is defined as 
\[\op{End}_{\bar{K}}(\phi):=\op{Hom}_{\bar{K}}(\phi, \phi)=\{u\in \bar{K}\{\tau\}\mid u \phi_a=\phi_a u\text{ for all }a\in A\}.\] This ring is in fact is an $A$-algebra, with $a\cdot u:=\phi_a u=u\phi_a$. Denote by 
\[\op{str}_\phi=\op{str}_{\phi, \bar{K}}: A\rightarrow \op{End}_{\bar{K}}(\phi)\] the \emph{structure homomorphism}, mapping $a$ to $\phi_a$. Since $\op{deg}_\tau(\phi_a)=r\op{deg}_T(a)$, it follows that $\op{str}_{\phi}$ is injective. One writes $\op{End}_{\bar{K}}(\phi)=A$ to mean that this structure homomorphism is an isomorphism. The following is a consequence of the \emph{open-image theorem} for Drinfeld-modules due to Pink and R\"utsche \cite{PinkRutsche}. 
\begin{theorem}[Pink-R\"utsche]\label{open image thm}
    Let $\phi$ be a Drinfeld module over $F$ such that $\op{End}_{\bar{F}}(\phi)=A$. Then, for any non-zero prime ideal $\p$ of $A$, the index $[\op{GL}_r(\widehat{A}_{\p}): \op{im}\left(\widehat{\rho}_{\phi, v_\p}\right)]$ is finite.
\end{theorem}

\section{Galois representations ramified at one prime}
\par In this section, I prove my main result. Fix an integer $r\geq 2$, a non-zero prime ideal $\p$ of $A$, and set $\cO$ to denote the completion $A_{\p}$. Let $f(T):=f_\p(T)$ be the monic irreducible polynomial that generates $\p$, and let $\phi$ be the Drinfeld module of rank $r$ with coefficients in $F$ defined by 
\[\phi_T:=T+\tau+f(T)\tau^r,\] and set $d:=\op{deg}_T(f(T))$.
\subsection{Good reduction away from $\{\p, \infty\}$}
\par In this brief subsection, I describe the reduction type of $\phi$ at all primes $\mathfrak{l}\notin \{\p, \infty\}$.
\begin{lemma}\label{good reduction lemma}
    The Drinfeld module $\phi$ has good reduction at all primes $\mathfrak{l}\notin \{\p, \infty\}$.
\end{lemma}
\begin{proof}
    Let $v\in \Omega_F$ be the valuation associated with $\mathfrak{l}$ and $\phi_v$ the completion of $\phi$ at $v$. I find that $\phi_{v}$ has all of its coeffcients in $R_v$ and the leading coefficient $f(T)$ is a unit in $R_v$. Therefore the reduction $\bar{\phi}_v$ is a Drinfeld module of rank $r$. This shows that $\phi$ has good reduction at $v$.
\end{proof}
  Let $\rho$ be the Galois representation
\[\rho=\widehat{\rho}_{\phi, \p}: \op{G}_F\rightarrow \op{GL}_r(\cO). \]
\begin{proposition}\label{prop 4.1}
 The Galois representation $\rho$ is unramified at all primes $\mathfrak{l}\notin \{\p, \infty\}$. 
\end{proposition}
\begin{proof}
    Let $a\in A$ be a monic generator of $\p$. I show that $\phi[a]$ is unramified at all primes $\mathfrak{l}\neq \p$. Let $\mathfrak{l}\neq \p$, and $v=v_{\mathfrak{l}}\in \Omega_F$ be the associated valuation. By Lemma \ref{good reduction lemma}, $\phi$ has good reduction at $\mathfrak{l}$. Therefore $\phi_{v}$ is isomorphic to a Drinfeld module $\psi$ over $R_v$, such that $\bar{\psi}$ is a Drinfeld module of rank $r$. Since $\bar{\psi}_a'(x)=a\neq 0$, it follows that $\bar{\psi}_a'(x)$ has distinct roots. Hensel's lemma implies that all roots of $\psi_a(x)$ lie in $\bar{\F}_q\cdot F_v$, the maximal unramified extension of $F_v$. Therefore, $\psi[a]$ is unramified. Since $\phi_v[a]\simeq \psi[a]$ as $\op{G}_{F_v}$-modules, it follows that $\phi_v[a]$ is unramified. In other words, $\phi[a]$ is unramified at $v$. 
\end{proof}

\subsection{The endomorphism ring}
Let $\kappa$ be a finite extension of $\F_q$, which is a quotient of $A$ by some non-zero prime ideal $\mathfrak{l}$. Let $C$ be the Carlitz module over $\kappa$. This is the Drinfeld module $C:A\rightarrow \kappa\{\tau\}$ defined by $C_T=t+\tau$, where $t$ is the image of $T$ of the quotient map $A\rightarrow A/\mathfrak{l}=\kappa$. 
\begin{proposition}\label{CarlitzProp}
    With respect to notation above, the structure map
    \[\op{str}_{C}: A\rightarrow \op{End}_{\bar{\kappa}}(C)\] is an isomorphism of $\F_q$-algebras.
\end{proposition}
\begin{proof}
    First, I show that $\op{str}_C$ is injective. Let $a\in A$ be in the kernel of $\op{str}_C$. Then, since $\op{deg}_\tau(C_a)=\op{deg}_T(a)$, it follows that $\op{deg}_T(a)=0$, i.e., $a\in \F_q$. On the other hand, $C:A\rightarrow \kappa\{\tau\}$ is a map of $\F_q$-algebras, hence, $a=0$. Next, I show that $\op{str}_C$ is surjective. Let $u=u_0+u_1\tau+\dots+u_m\tau^m\in \bar{\kappa}\{\tau\}$ be an element in $\op{End}_{\bar{\kappa}}(\phi)$. I show that for some $b\in A$, one has that $u=C_b$. Assume that $u_m\neq 0$, i.e., $m=\op{deg}_\tau(u)$. I prove the result by induction on $m$. First, when $m=0$, $u=u_0$, i.e., $u\in \kappa$. Then, the relation
    \[u_0 C_T=C_T u_0,\] implies
    \[\begin{split}& u_0(T+\tau)=(T+\tau)u_0, \\ 
    \text{i.e.,} & u_0 T+u_0 \tau=u_0 T+u_0^q \tau.\end{split}\]
    This implies that $u_0=u_0^q$, i.e. $u=u_0\in \F_q$. Since $C$ is a homomorphism of $\F_q$-algebras, I find that $u=C_{u_0}$. Thus, one may assume that the assertion holds for all $v\in \op{End}_k(C)$ such that $\op{deg}_\tau(v)<m$. Comparing the highest degree terms for the relation 
    \[uC_T=C_T u\] I have that
    \[u_m \tau^{m+1}+\text{ lower degree terms }=\tau u_m\tau^m +\text{ lower degree terms }.\] Thus, I find that 
    \[u_m \tau^{m+1}=\tau u_m \tau^m=u_m^q \tau^{m+1}.\] Therefore, one finds that $u_m\in \F_q$. I deduce that 
    \[v:=u-C_{u_m T^m}=u-u_m(T+\tau)^m\] satisfies the property that $\op{deg}_{\tau}(v)<m$. By inductive hypothesis, $v$ is in the image of $\op{str}_C$, and hence, $u$ is in the image of $\op{str}_C$.
\end{proof}

Let $K$ be the completion of $F$ at a prime $\p$ of $A$. Let $\mathbb{C}_K$ be the completion of the algebraic closure of $K$. 
\begin{definition}
    Let $\phi$ be a Drinfeld module over $K$, an $A$-lattice of rank $d$ is a free $A$-module of rank $d$ contained in $^{\phi}K^{\op{sep}}$ which is $\op{G}_K$-stable and discrete in $\mathbb{C}_K$.
\end{definition}

\begin{theorem}[Drinfeld-Tate Uniformization]\label{DF uniformization}
     Let $r_1>0, r_2> 0$ be integers and $r:=r_1+r_2$. There is a natural bijection
     \[\begin{split}\{(\phi, \Lambda)\mid & \phi \text{ is a Drinfeld module over }R\text{ with good reduction and of rank }r_1, \\
     & \text{ and }\Lambda\text{ is a }\phi\text{ lattice of rank }r_2\}\\
     \xrightarrow{\sim}\{\psi \mid & \psi\text{ is a Drinfeld module of rank }r\text{ over }R\text{ with stable reduction} \\
     & \text{ and reduction rank }r_1\} \end{split}\]
     and, $\bar{\phi}$ is equal to $\bar{\psi}$. The pair $(\phi, \Lambda)$ is referred to as the uniformization of $\psi$. 
\end{theorem}
\begin{proof}
    For a proof of the above result, see \cite[section 7]{DrinfeldEllipticModules} or \cite[Theorem 6.2.11]{papibook}.
\end{proof}
\begin{proposition}\label{prop 4.3}
    Let $\psi_1$ and $\psi_2$ be Drinfeld modules of rank $r$ over $R$ with stable reduction. Assume that $\op{rank}(\bar{\psi}_1)=\op{rank}(\bar{\psi}_2)$. Let $(\phi_i, \Lambda_i)$ be the Drinfeld-Tate uniformization of $\psi_i$. Then, there is an injection of $A$-modules
    \[\op{Hom}_{\bar{K}}(\psi_1, \psi_2)\hookrightarrow \op{Hom}_{\bar{K}}(\phi_1, \phi_2).\]
\end{proposition}
\begin{proof}
    The result follows from \cite[Theorem 6.2.12]{papibook}.
\end{proof} 

\begin{corollary}\label{strK iso}
    Let $r\geq 2$ be an integer, $\p$ be a non-zero prime ideal of $A$ and $\cO:=A_{\p}$. Let $f(T)$ be the monic polynomial generator of $\p$ and $\phi$ be the Drinfeld module of rank $r$ with coefficients in $F$ defined by 
\[\phi_T:=T+\tau+f(T)\tau^r.\] Then, the structure map 
\[\op{str}_\phi: A\rightarrow \op{End}_{\bar{F}}(\phi)\] is an isomorphism.
\end{corollary}
\begin{proof}
    At the prime $\p$, the Drinfeld module has stable reduction with reduction rank $1$. Let $\kappa$ denote the residue field of $A$ at $\p$. Take $\psi$ to be the completion of $\phi$ at $\p$. I find that $\psi$ is defined over $\cO$ and that its reduction mod-$\p$ is equal to $C$ (the Carlitz module over $\kappa$). Set $K$ to denote the fraction field of $\cO$. Let $(\psi', \Lambda)$ be the Drinfeld-Tate uniformization of $\psi$. It follows from the Theorem \ref{DF uniformization} that $\bar{\psi}'=\bar{\psi}$, and hence, $\bar{\psi}'=C$. It follows from Proposition \ref{CarlitzProp} that the structure map \[\op{str}_{C}: A\rightarrow \op{End}_{\bar{\kappa}}(C)\] is an isomorphism of $A$-modules. This structure map factors as a composite of natural maps
    \[A\xrightarrow{\alpha} \op{End}_{\bar{K}}(\psi)\xrightarrow{\beta} \op{End}_{\bar{K}}(\psi')\xrightarrow{\gamma} \op{End}_{\bar{\kappa}}(C).\] The map $\alpha$ is the structure map $\op{str}_{\psi, K}$. The map $\beta$ is injective by Proposition \ref{prop 4.3}. The map $\gamma$ is the reduction map and is injective by \cite[Corollary 6.1.12]{papibook}. Since $\gamma \beta\alpha=\op{str}_C$ is surjective and $\gamma \beta$ is injective, it follows that $\alpha$ is surjective. On the other hand, since $\gamma \beta\alpha=\op{str}_C$ is injective, it follows that $\alpha$ is injective. Hence the map $\op{str}_\psi=\alpha:A\rightarrow \op{End}_{\bar{K}}(\psi)$ is an isomorphism. This map factors as a composite
    \[A\rightarrow  \op{End}_{\bar{F}}(\phi)\rightarrow \op{End}_{\bar{K}}(\phi),\] the latter map is injective. It follows that $\op{str}_\phi$ is an isomorphism of $A$-modules. 
\end{proof}

\subsection{Ramification at $\p$}
\par In this subsection, I analyze the reduction type of $\phi$ at $\p$ and the image of the local Galois representation at $\p$. 
\begin{proposition}\label{prop 4.4}
    Let $r\geq 2$ be an integer, $\p$ be a non-zero prime ideal of $A$ and $\cO:=A_{\p}$. Let $f(T)$ be the monic polynomial generator of $\p$ and $\phi$ be the Drinfeld module of rank $r$ with coefficients in $F$ defined by 
\[\phi_T:=T+\tau+f(T)\tau^r.\] The following assertions hold
\begin{enumerate}
    \item $\op{End}_{\bar{F}}(\phi)=A$,
    \item the image of $\rho$ has finite index in $\op{GL}_r(\cO)$,
    \item  the restriction of $\rho$ to the inertia group $\op{I}_{\p}$ is of the form
        \[\rho_{|\op{I}_{\p}}=\begin{pmatrix}
  \chi & \ast & \ast & \cdots & \ast & \ast \\
   & 1 & 0 & \cdots & 0 & 0\\
    & & 1 & 0 & \cdots & 0 \\
    & &  & \ddots & \vdots & \vdots \\  
    & &  &    & 1 & 0 \\
    & &  &    &  & 1 &   \\
   \end{pmatrix},\] where $\chi$ is a totally ramified character.
\end{enumerate}
\end{proposition}
\begin{proof}
    Corollary \ref{strK iso} asserts that \[\op{str}_{\phi}: A\rightarrow \op{End}_{\bar{F}}(\phi)\] is an isomorphism. That the image of $\rho$ has finite index in $\op{GL}_r(\cO)$ follows from Theorem \ref{open image thm}. The reduction $\bar{\phi}$ of $\phi$ at $\p$ is equal to the Carlitz module, and hence $\bar{\phi}$ is of height $1$. It then follows that there is an exact sequence of $A[\op{G}_{F_{\p}}]$-modules
    \[0\rightarrow \phi^0[\p]\rightarrow \phi[\p]\rightarrow \bar{\phi}[\p]\rightarrow 0,\]
    where $\phi^0[\p]:=\{\alpha\in \phi[\p](\bar{F}_{\p})\mid |\alpha|<1\}$. As $A$-modules, \[\phi^0[\p]\simeq (A/\p)^{\op{H}(\bar{\phi})}=(A/\p),\] and 
    \[\bar{\phi}[\p]\simeq (A/\p)^{r-\op{H}(\bar{\phi})}=(A/\p)^{r-1}.\] Moreover, the module $\phi^0[\p]$ (resp. $\bar{\phi}[\p]$) is totally ramified (resp. unramified). For a proof of these statements, the reader may refer to \cite[section 6.3]{papibook}. Taking $T_{\p}(\phi)^0:=\varprojlim_n \phi[\p]^0$, I have a short exact sequence of $A_{\p}[\op{G}_{F_{\p}}]$-modules
    \[0\rightarrow T_{\p}(\phi)^0\rightarrow T_{\p}(\phi)\rightarrow T_{\p}(\bar{\phi})\rightarrow 0,\] where 
    \[T_{\p}(\phi)^0\simeq A_{\p}\text{ and }T_{\p}(\bar{\phi})\simeq A_{\p}^{r-1}.\] The Galois action on $T_{\p}(\bar{\phi})$ is unramified. Thus, I find that the restriction $\rho_{|\op{I}_{\p}}$ is of the desired form.
\end{proof}
\subsection{Ramification at $\infty$}
\par In this subsection, I show that $\rho$ is indeed unramified at $\infty$. This part of the argument does require that I assume that $\deg \p=1$. Recall that $F_\infty$ is the completion of $F$ at $\infty$. Fix an algebraic closure of $\bar{F}_{\infty}/F_\infty$ and think of every algebraic extension of $F_\infty$ as a subfield of $\bar{F}_\infty$. For $k\in \Z_{\geq 0}$, let $L_k/F_\infty$ be the splitting field of the polynomial $g_k(x):=\phi_{f(T)^k}(x)$, with the understanding that $L_0:=F_\infty$. Let $v=v_\infty$ be the valuation at $\infty$ normalized by $v(T)=-1$, and note that that $v(\cdot)=-\op{deg}(\cdot)$. Since it is assumed that $\op{deg}f=1$, one may write $f(T)=(T-c)$, where $c\in \F_q$. 

\begin{lemma}\label{unramified lemma}
    With respect to notation above, one has that for all $k\geq 1$, the extension $L_k/F_\infty$ is unramified. 
\end{lemma}
\begin{proof}
    I prove the result by induction on $k$, starting with $k=1$. Let $v_1$ be the valuation on $L_1$ which extends $v$. Let \[G(x):=(T-c)^{-1} \phi_{(T-c)}(x)=x+(T-c)^{-1}x^q+x^{q^r}\] and $\bar{G}(x)$ be the reduction of $G(x)$ modulo the maximal ideal of $\cO_\infty=\F_q\llbracket T^{-1}\rrbracket$. It is clear that for all $k\geq 1$, $L_k$ is the splitting field of $G(x)$ over $L_{k-1}$. It is easy to see that $\bar{G}(x)=x+x^{q^r}$ and $\bar{G}'(x)=1$. Thus, $\bar{G}(x)$ has the same number of roots as $G(x)$ and its roots are all distinct. By Hensel's lemma, the roots of $G(x)$ lie in $\bar{\F}_q \cdot F_\infty$, the maximal unramified extension of $F_\infty$. Therefore, in particular, $L_1/F_\infty$ is unramified. I argue by induction and assume that $L_{k-1}/F_\infty$ is unramified. Then, the valuation $v$ is unchanged when extended to $L_{k-1}$. Note that $L_k$ is the splitting field of $G(x)$ over $L_{k-1}$. Thus, the same argument gives us that $L_{k}/L_{k-1}$ is unramified. Thus, $L_k$ is an unramified extension of $F_\infty$, and I conclude that $\rho$ is unramified at $\infty$ when $\op{deg} \p=1$. 
\end{proof}

\begin{remark}
    Alternatively, Lemma \ref{unramified lemma} can be proven via analyzing the Newton polygon of $G(x)$. 
\end{remark}

\begin{proposition}\label{unramified propn}
    Assume that $\deg \p=1$. Then, $\rho$ is unramified at $\infty$. 
\end{proposition}
\begin{proof}
Choose an embedding of $\bar{F}$ into $\bar{F}_\infty$. This is equivalent to the choice of a prime of $\bar{F}$ that lies above $\infty$. It is clear that the field $F_\infty(\phi[\p^k])$ coincides with the splitting field of $g_k(x)$. From Lemma \ref{unramified lemma}, I find that $F_\infty(\phi[\p^\infty]):=\bigcup_k F_\infty(\phi[\p^k])=\bigcup_k L_k$ is an unramified extension of $F_\infty$. Consider the local Galois representation \[\rho_{|\infty}: \op{Gal}(\bar{F}_\infty/F_\infty)\rightarrow \op{GL}_r(\cO)\]and observe that $F_\infty(\phi[\p^\infty])$ is simply the field fixed by the kernel of $\rho_{|\infty}$. Thus, the inertia group at $\infty$ is contained in the kernel of $\rho_{|\infty}$, or in other words, $\rho$ is unramified at $\infty$.
\end{proof}

\subsection{Proof of the main theorem}

\begin{proof}[Proof of Theorem \ref{main thm}]
    It follows from Proposition \ref{prop 4.1} that $\rho$ is unramified at all primes $\mathfrak{l}\notin \{\p, \infty\}$, and from Proposition \ref{prop 4.4} that $\op{im}\rho$ has finite index in $\op{GL}_r(\cO)$ and that the restriction $\rho_{|\op{I}_{\p}}$ is of the desired form. When $\deg \p=1$, Proposition \ref{unramified propn} asserts that $\rho$ is also unramified at $\infty$.
\end{proof}

\subsection*{Declarations}
\begin{description}
    \item[Ethics approval and consent to participate] Not applicable, no animals were studied or involved in the preparation of this manuscript.
    \item[Consent for publication] The authors give their consent for the publication of the manuscript under review.
    \item[Availability of data and materials] No data was generated or analyzed in obtaining the results in this article.
    \item[Competing interests] There are no competing interests to report.
    \item[Funding] There are no funding sources to report.
    \item[Author contribution] The manuscript was prepared by the author and is wholly the author's contribution.
    \item[Acknowledgements] Not applicable.
\end{description}

\bibliographystyle{alpha}
\bibliography{references}
\end{document}